\documentclass[11pt, english, draft]{article}
\usepackage{amssymb,amsmath}
\title{\bf Recognizing by Spectrum for the Automorphism Groups of Sporadic Simple
Groups }
\author{ {\bf V. D. Mazurov}\thanks{This work is supported by Russian Science
Foundation (Project no. 14-21-00065).} \  and  \ {\bf A. R.
Moghaddamfar}}

\newenvironment{proof}{\noindent {\em {Proof}}.}{$\square$
\medskip}

\newtheorem{theorem}{Theorem}

\newtheorem{lm}{Lemma}

\begin{document}

\maketitle

\begin{abstract}
\noindent The spectrum of a finite group is the set of its
element orders, and two groups are said to be isospectral if they
have the same spectra. A finite group $G$ is said to be
recognizable by spectrum, if every finite group isospectral with
$G$ is isomorphic to $G$. We prove that if $S$ is any of the
sporadic simple groups $M^cL$, $M_{12}$, $M_{22}$, $He$, $Suz$,
$O'N$, then ${\rm Aut}(S)$ is recognizable by spectrum. This
finishes the proof of the recognizability by spectrum of the
automorphism groups of all sporadic simple groups, except $J_2$.
Furthermore, we show that if $G$ is isospectral with ${\rm
Aut}(J_2)$, then either $G$ is isomorphic to ${\rm Aut}(J_2)$, or
$G$ is an extension of a $2$-group by $\mathbb{A}_8$.
\end{abstract}

\renewcommand{\baselinestretch}{1.1}
\def\thefootnote{ \ }

\footnotetext{{\em $2000$ Mathematics Subject Classification}:
$20D05$.}

\section{Introduction}
The {\em spectrum} $\omega(G)$ of a finite group $G$ is the set of
all element orders of $G$. This set is a nonempty subset of the
set of natural numbers and is closed under divisibility. Hence,
$\omega(G)$ is uniquely determined by the subset $\mu(G)$ of its
elements which are maximal under the divisibility relation. For an
arbitrary subset $\omega$ of the set of natural numbers, denote by
$h(\omega)$ the number of isomorphic classes of finite groups $G$
such that $\omega(G)=\omega$. For convenience, we will write
$h(G)$ instead of $h(\omega(G))$. A finite group $G$ is said to be
{\em recognizable by spectrum} (shortly, {\em recognizable}) if
$h(G)=1$, {\em almost recognizable} if $1<h(G)<\infty$, and {\em
nonrecognizable} if $h(G)=\infty$.

The set $\omega(G)$ determines the {\em prime graph} (or {\em
Gruenberg-Kegel} graph) ${\rm GK}(G)$ whose vertex set is
$\pi(G)$, the set of primes dividing the order of $G$, and two
distinct vertices $p$ and $q$ are {\em adjacent} if and only if
$pq\in \omega(G)$. We denote by $s(G)$ the number of connected
components of ${\rm GK}(G)$ and by $\pi_i(G)$, $i= 1, 2, \ldots,
s(G)$, the set of vertices of $i$-th connected component. We
often identify a connected component of a prime graph with its
vertex set, and vice versa. If $2\in \pi(G)$, then we assume that
$2\in \pi_1(G)$. The vertex set of connected components of prime
graphs associated with finite simple groups are listed in \cite{k}
and \cite{w} (see the improved list in \cite{newmaz}).

In (\cite{lishi}, \cite{mazshi},
\cite{praegershi}--\cite{wujieli}), it is proved that the sporadic
simple groups except $J_2$ are recognizable and $J_2$ is
nonrecognizable. In \cite{mzd}, it is showed that if $S$ is a
sporadic simple group except $M_{12}$, $M_{22}$, $J_2$, $He$,
$Suz$, $M^{c}L$ and $O'N$ then ${\rm Aut}(S)$ is recognizable by
spectrum, and if $S$ is isomorphic to $M_{12}$, $M_{22}$, $He$,
$Suz$ or $O'N$, then $h({\rm Aut}(S))\in \{1,\infty\}$.

For convenience, we list in Table 1 the sets $\mu({\rm Aut}(S))$
and the connected components of ${\rm Aut}(S)$ for these sporadic
simple groups.

The goal of this paper is to prove the following
theorems:\\[0.3cm]
{\bf Theorem A.} {\em If $S$ is isomorphic to $M^cL$, $M_{12}$,
$M_{22}$, $He$, $Suz$ or $O'N$, then ${\rm Aut}(S)$ is
recognizable by spectrum.}\\[0.2cm]
{\bf Theorem B.} {\em Let $G$ be a finite group with
$\omega(G)=\omega({\rm Aut}(J_2))$. Then, either $G$ is isomorphic
to ${\rm Aut}(J_2)$, or $G$ is an extension of a $2$-group by
$\mathbb{A}_8$.}\\[0.2cm]
{\em Remark.} In the second case of Theorem B, we have not been
able to find any group $G$ with the given properties.

Therefore, Theorem A furnishes the recognizability question for
the automorphism groups of all sporadic simple groups except
$J_2$. In fact, Theorem A and the results of \cite{mzd} imply the
following.\\[0.3cm]
{\bf Corollary.} {\em  The automorphism group of every sporadic
simple groups, except $J_2$, is recognizable by spectrum.}\\[0.3cm]
\indent We use the following notation. For a finite group $G$,
denote by $\pi(G)$ the set of all prime divisors of $|G|$. In
addition, let ${\rm Soc}(G)$ be the socle of $G$, i.e., the
product of all minimal normal subgroups of $G$. Denote by
$\mathbb{A}_n$ and $\mathbb{S}_n$ the alternating and symmetric
group on $n$ letters,  respectively.
\begin{center}
{\bf Table 1.} {\em The spectrum of automorphism group of some
sporadic simple
groups.}\\[0.2cm]
\begin{tabular}{|l|l|l|l|}
\hline
$S$  & $\mu({\rm Aut}(S))$ & $\pi_1({\rm Aut}(S))$ & $\pi_2({\rm Aut}(S))$\\
\hline
$M_{12}$  & 8, 10, 11, 12 & 2, 3, 5 & 11\\
$M_{22}$ & 8, 10, 11, 12, 14& 2, 3, 5, 7& 11 \\
$J_2$ &  10, 14, 15, 24 & 2, 3, 5, 7 & $-$ \\
$He$ &  16, 17, 20, 24, 28, 30, 42& 2, 3, 5, 7& 17 \\
$M^cL$ & 9, 14, 20, 22, 24, 30   & 2, 3, 5, 7, 11 & $-$ \\
$Suz$  & 13, 16, 18, 21, 22, 24, 28, 30, 40& 2, 3, 5, 7, 11 & 13\\
$O'N$  & 16, 20, 22, 24, 30, 31, 38, 56& 2, 3, 5, 7, 11, 19& 31 \\
\hline
\end{tabular}\\[0.7cm]
\end{center}

\section{Preliminary Results}

\begin{lm}\label{lemma0}
Let $K$ be an Abelian normal subgroup of a group $G$ and
$\varphi: G\rightarrow {\rm Aut}(K)$ be the homomorphism induced
by conjugation in $G$. Suppose that $Kg$ is an element of order
$m$ in $G/K$ and $g^m=1$. All elements in the coset $Kg$ are of
order $m$ if and only if the equality $f(g^\varphi)=0$ holds in
the endomorphism ring of $A$ where $f(x)=1+x+\cdots+x^{m-1}$.
\end{lm}
\begin{proof}
This is the obvious consequence of the well-known equality:
$$(ag)^m=g^m\cdot a^{g^m}\cdot a^{g^{m-1}}\cdots a,$$ where $a\in
K$. \end{proof}

\begin{lm}\label{automorphism}
Let $S=P_1\times P_2\times \cdots \times P_t$, where $P_i$'s are
isomorphic non-Abelian simple groups. Then, there hold $${\rm
Aut}(S)\cong \left({\rm Aut}(P_1)\times {\rm Aut}(P_2)\times
\cdots \times {\rm Aut}(P_t)\right)\rtimes \mathbb{S}_t.$$ In
particular, $|{\rm Aut}(S)|=\prod_{i=1}^t|{\rm Aut}(P_i)|\cdot
t!.$
\end{lm}
\begin{proof} See Lemma 2.2 in \cite{zavarnitsine}.
\end{proof}

\begin{lm}\label{vicmazurov} Let $G$ be a finite group, and let
$N\lhd G$ and $G/N$ be a Frobenius group with kernel $F$ and
cyclic complement $C$. If $(|F|,|N|)=1$ and $F$ does not lie in
$NC_G(N)/N$, then $p|C|\in \omega(G)$ for some prime divisor $p$
of $|N|$. Furthermore, if the preimage of $F$ in $G$ is a
Frobenius group, then $$|C|\cdot\prod_{p\in \pi(N)}p\in
\omega(G)$$
\end{lm}
\begin{proof} See Lemma 1 in \cite{maz}.
\end{proof}
\section{Recognition of ${\rm Aut}(M^cL)$ }

Let $p$ be a prime. Denote by $\mathcal{S}_{p}$ the set of all
finite non-Abelian simple groups whose prime divisors are at most
$p$. In this section we deal with the list of finite non-Abelian
simple groups in $\mathcal{S}_{11}$, determined in \cite{maz} (see
Table 2 below). In particular, the information from this table
implies:

\begin{lm}\label{elementary}
If $S\in \mathcal{S}_{11}$, then $\{2,3\}\subseteq \pi(S)$ and
$\pi({\rm Out}(S))\subseteq \{2,3\}$.
\end{lm}

\begin{center}
{\bf Table 2}. {\em Finite non-Abelian simple groups $S\in \mathcal{S}_{11}$}. \\[0.3cm]
\begin{tabular}{|l|l|l||l|l|l|}
\hline $S$ & Order of $S$ & ${\rm Out}(S)$&$S$ & Order of $S$ & ${\rm Out}(S)$\\[0.1cm]
\hline $\mathbb{A}_5$ & $2^2\cdot 3\cdot 5$ &  2 & $\mathbb{A}_9$ & $2^6\cdot 3^4\cdot 5\cdot 7$ &  2\\[0.1cm]
$L_2(7)$& $2^3\cdot 3\cdot 7$ &  2 & $M_{22}$ & $2^7\cdot 3^2\cdot 5\cdot 7\cdot 11$ & 2\\[0.1cm]
$\mathbb{A}_6\cong L_2(9)$ & $2^3\cdot 3^2\cdot 5$ & $2^2$ & $J_2$ & $2^7\cdot 3^3\cdot 5^2\cdot 7$ & 2\\[0.1cm]
$L_2(8)$ & $2^3\cdot 3^2\cdot 7$ & 3 & $S_6(2)$ & $2^9\cdot 3^4\cdot 5\cdot 7$ & 1\\[0.1cm]
$L_2(11)$ & $2^2\cdot 3\cdot 5\cdot 11$ & 3 & $\mathbb{A}_{10}$ & $2^7\cdot 3^4\cdot 5^2\cdot 7$ & 2\\[0.1cm]
$\mathbb{A}_7$ & $2^3\cdot 3^2\cdot 5\cdot 7$ &  2 & $U_4(3)$ & $2^7\cdot 3^6\cdot 5\cdot 7$ & $D_8$\\[0.1cm]
$U_3(3)$ & $2^5\cdot 3^3\cdot 7$ & 2 & $U_5(2)$ & $2^{10}\cdot 3^5\cdot 5\cdot 11$ & 2\\[0.1cm]
$M_{11}$ & $2^4\cdot 3^2\cdot 5\cdot 11$ & 1 & $\mathbb{A}_{11}$ & $2^7\cdot 3^4\cdot 5^2\cdot 7\cdot 11$ & 2 \\[0.1cm]
$\mathbb{A}_8$& $2^6\cdot 3^2\cdot 5\cdot 7$ &  2 & ${\rm HS}$ & $2^9\cdot 3^2\cdot 5^3\cdot 7^2\cdot 11$ & 2\\[0.1cm]
$L_3(4)$ & $2^6\cdot 3^2\cdot 5\cdot 7$& $D_{12}$ & $S_4(7)$ & $2^8\cdot 3^2\cdot 5^2\cdot 7^4$& 2\\[0.1cm]
$U_4(2)$& $2^6\cdot 3^4\cdot 5$ & 2 & $O_8^+(2)$ & $2^{12}\cdot 3^5\cdot 5^2\cdot 7$ & $S_3$\\[0.1cm]
$L_2(49)$ & $2^4\cdot 3\cdot 5^2\cdot 7^2$ & $2^2$ & $\mathbb{A}_{12}$ & $2^9\cdot 3^5\cdot 5^2\cdot 7\cdot 11$ & 2\\[0.1cm]
$M_{12}$ & $2^6\cdot 3^3\cdot 5\cdot 11$ & 2 & $M^cL$ & $2^7\cdot 3^6\cdot 5^3\cdot 7\cdot 11$ & 2\\[0.1cm]
$U_3(5)$ & $2^4\cdot 3^2\cdot 5^3\cdot 7$ & $S_3$ & $U_6(2)$ & $2^{15}\cdot 3^6\cdot 5^2\cdot 7\cdot 11$ & $S_3$\\[0.1cm]
 \hline
\end{tabular}
\end{center}

\begin{theorem}\label{th1}
The automorphism group of $M^cL$ is recognizable by spectrum.
\end{theorem}

{\em Proof.}  Let $A={\rm Aut}(M^cL)$ and let $G$ be a finite
group such that $$\mu(G)=\mu(A)=\{3^2, \ 2\cdot 7, \ 2^2\cdot 5, \
2 \cdot 11, \ 2^3\cdot 3, \ 2\cdot 3\cdot 5\}.$$ It is obvious
that $s(G)=1$. The proof of Theorem 1 will be done through a
sequence of lemmas.
\begin{lm}\label{l1mac}
Let $N$ be the maximal normal soluble subgroup of $G$. Then only
one of the three primes $5$, $7$, $11$ can divide the order of
$N$.
\end{lm}
\begin{proof}
Assume first that $\{5, 7, 11\}\subseteq \pi(N)$. Since every two
distinct primes in $\{5,7,11\}$ are nonadjacent in ${\rm GK}(G)$,
the same is true for ${\rm GK}(N)$. Now, by Lemma 8 in \cite{lm},
$N$ is insoluble, which is a contradiction.

Let $p, q$ and $r$ be distinct primes in $\{5,7,11\}$ given in
arbitrary order. Assume that two of them, for definiteness $p$ and
$q$, divide $|N|$, whereas $r$ does not. Consider a Hall
$\{p,q\}$-subgroup $T$ in $N$. By Frattini argument, $G=NN_G(T)$.
Therefore, the normalizer $N_G(T)$ contains an element of order
$r$, which acts fixed-point-freely on $T$. Thompson Theorem
implies that $T$ is nilpotent. Hence $p\cdot q\in
\omega(T)\subseteq \omega(G)$, which is a contradiction.
\end{proof}
\begin{lm}\label{l2mac} There exists a finite simple group $S\in \mathcal{S}_{11}$
such that $S\leqslant G/N\leqslant {\rm Aut}(S)$.
\end{lm}
\begin{proof}
Let $\overline{G}=G/N$. Then $S={\rm Soc}(\overline{G})=P_1\times
P_2\times \cdots \times P_k$, where $P_i$ are non-Abelian simple
groups and $\overline{G}\leqslant {\rm Aut}(S)$. It is obvious
that $P_i\in \mathcal{S}_{11}$, so we need only to prove that
$k=1$.

Suppose $k\geqslant 2$. By Lemma \ref{elementary}, $|P_i|$ is
divisible by $3$. By Lemma \ref{l1mac}, there exists a prime
$p\in \{7, 11\}$ which divides the order of $\overline{G}$. It is
clear that neither $7$ nor $11$ can divide $|S|$, since otherwise
$7\cdot 3$ or $11\cdot 3 \in \omega(S)\subseteq \omega(G)$, which
is a contradiction. Hence $\pi(S)=\pi(P_i)=\{2,3,5\}$. Thus, $p$
divides the order of ${\rm Out}(S)$. But ${\rm Out}(S)={\rm
Out}(S_1)\times \cdots\times {\rm Out}(S_m)$, where the groups
$S_j$ are the direct products of those $P_i$ which are isomorphic
and $S\cong S_1\times \cdots \times S_m$. Therefore, for some $j$,
$p$ divides the order of an outer automorphism group of a direct
product $S_j$ of $t$ isomorphic simple groups $P_i$. By Lemma
\ref{elementary}, the order of ${\rm Out}(P_i)$ is not divisible
by $p$ . By Lemma \ref{automorphism}, $|{\rm Aut}(S_j)|=|{\rm
Aut}(P_i)|^t\cdot t!$. Therefore, $t\geqslant p$ and $S_j$ admits
an automorphism of order $p$ which centralizes the element
$(a,a,\ldots,a)$ of $S_j$ of order $3$, where $a$ is any element
of $P_i$ of order $3$. Thus $3\cdot p\in \omega(G)$, a
contradiction.
\end{proof}
\begin{lm}\label{l3mac} $S\cong M^cL$.
\end{lm}
\begin{proof} We consider all possibilities for the group $S$
consecutively.
\begin{itemize}
\item[{(1)}]
$S$ is isomorphic to $\mathbb{A}_5$, $L_2(7)$, $\mathbb{A}_6\cong
L_2(9)$, $L_2(8)$, $U_3(3)$ or $U_4(2)$. Since the order of ${\rm
Out}(S)$ is not divisible by $5$, $7$, $11$ and only one of these
primes divides the order of $S$, we have a contradiction by Lemma
\ref{l1mac}.

\item[{(2)}] $S$ is isomorphic to $L_2(49)$ or $S_4(7)$. In
this case
 $25\in\omega(S)\backslash \omega(G)$; a contradiction.

\item[{(3)}] $S$ is isomorphic to $\mathbb{A}_n$, $n=10$, $11$
or $12$. Since $3\cdot 7\in \omega(S)\backslash \omega(G)$, we
have a contradiction.

\item[{(4)}] $S$ is isomorphic to $U_5(2)$ or $U_6(2)$. Since
$18\in\omega(S)\backslash \omega(G)$, which is a contradiction.

\item[{(5)}] $S$ is isomorphic to $L_3(4)$, $U_3(5)$, $J_2$,
$S_6(2)$, $U_4(3)$,
 $O_8^+(2)$ or $\mathbb{A}_n$, $n=7$, $8$ or $9$. Since
$11$ does not divide the order of ${\rm Aut}(S)$, we have $11\in
\pi(N)$. On the other hand, each of these groups contains a
Frobenius group of order $21$. By Lemma \ref{vicmazurov}, there
exists an element of order $11\cdot 3$ in $G$; a contradiction.

\item[{(6)}] $S$ is isomorphic to $L_2(11)$, $M_{11}$ or
$M_{12}$. Since $7$ does not divide the order of ${\rm Aut}(S)$,
we have $7\in \pi(N)$. Moreover, each of these groups contains a
Frobenius subgroup of order $55$. Now, by Lemma \ref{vicmazurov},
there exists an element of order $7\cdot 5$ in $G$; a
contradiction.

\item[{(7)}] $S$ is isomorphic to ${\rm HS}$ or $M_{22}$. Since
${\rm Aut}(S)$ does not contain an element of order $9$,
therefore, $|N|$ is divisible by $3$. On the other hand, $S$
contains a Frobenius subgroup of order $8\cdot 7$, and so $3\cdot
7$ belongs to $\omega(G)$, which is impossible.
\end{itemize}

Thus $S\cong M^cL$.
\end{proof}
\begin{lm}\label{l4mac} $N=1$.
\end{lm}
\begin{proof} Suppose $N\neq 1$. As in Lemma 3.4 of \cite{vasilev}, we may assume that
$N$ is a nontrivial elementary Abelian $p$-group where $p=2, 3, 5,
7$ or $11$. Since $S$ contains a Frobenius subgroup of order
$8\cdot 7$, Lemma \ref{vicmazurov} implies that $p\neq 3, 5, 7,
11$.

Let $p=2$. If $M^cL$ acts trivially on $N$, then $G$ contains an
element of order $18$ which is not the case. If $M^cL$ acts
faithfully, then $G$ contains an element of order $16$ since
$M^cL$ contains a Frobenius group $3^2:8$ which lies in
$3^6:M_{10}$ of $M^cL$.
\end{proof}
\begin{lm}\label{l5mac} $G\cong {\rm Aut}(M^cL)$.
\end{lm}
\begin{proof}
By Lemmas \ref{l2mac}, \ref{l3mac}, \ref{l4mac}, $G\cong M^cL$ or
$G\cong {\rm Aut}(M^cL)$, and since $20\in \omega(G)\backslash
\omega(M^cL)$, we deduce that $G\cong {\rm Aut}(M^cL)$. This
proves both the lemma and the theorem.
\end{proof}

\section{Recognition of ${\rm Aut}(M_{12})$, ${\rm Aut}(M_{22})$, ${\rm Aut}(He)$, ${\rm
Aut}(Suz)$, ${\rm Aut}(O'N)$ }

\begin{theorem}
If $S$ is isomorphic to one of the sporadic simple groups
$M_{12}$, $M_{22}$, $He$, $Suz$ or $O'N$, then ${\rm Aut}(S)$ is
recognizable by spectrum.
\end{theorem}
\begin{proof} Let $G$ be a finite group and $\omega(G)=\omega({\rm
Aut}(S))$. By Theorems 2 and 3 in \cite{mzd}, $G$ has a normal
$2$-subgroup $N$ such that $G/N\cong S$. Now we prove that $N=1$.
Suppose false. One can assume that $N$ is elementary Abelian and
absolutely irreducible as $S$-module. If $S$ is isomorphic to
$He$, $Suz$ or $O'N$, then the table of $2$-Brauer characters of
$S$ shows that $G$ contains an element of order $2\cdot 17$,
$2\cdot 13$ or $2\cdot 31$, respectively, which is a
contradiction. If $S$ is isomorphic to $M_{12}$ or $M_{22}$, then
since ${\rm Aut}(S)$ does not contain an element of order $2\cdot
11$, the table of Brauer characters shows that $N$ as $S$-module
has dimension $10$, but then the Atlas of finite group
representations and Lemma \ref{lemma0} show that some pre-image in
$G$ of an element contained in the class $8A$ of $S$ is of order
$16\notin \omega({\rm Aut}(S))$. Thus, $N=1$ and $G\cong {\rm
Aut}(S)$.
\end{proof}
\section{On Recognition of ${\rm Aut}(J_2)$}

\begin{theorem}\label{th2}
Let $G$ be a group such that $\omega(G)=\omega({\rm Aut}(J_2))$.
Then, either $G\simeq {\rm Aut}(J_2)$, or $G$ has a normal
$2$-subgroup $N$ such that $G/N\simeq \mathbb{A}_8$.
\end{theorem}
{\em Proof.}  Let $A={\rm Aut}(J_2)$ and let $G$ be a finite group
with
$$\mu(G)=\mu(A)=\{2\cdot 5, \ 2\cdot 7, \ 3\cdot 5, \
2^3\cdot 3\}.$$ Evidently, $s(G)=1$. We divide the proof into a
number of separate lemmas. Let $N$ be the maximal normal soluble
subgroup of $G$.
\begin{lm}\label{l2janko}  There are some finite simple groups $S\in
\mathcal{S}_{7}$ with $7\in \pi(S)$ such that $S\leqslant G/N\leq
{\rm Aut}(S)$.
\end{lm}
\begin{proof}
Let $\overline{G}=G/N$. Then $S={\rm Soc}(\overline{G})=P_1\times
P_2\times \cdots \times P_k$, where $P_i$ are non-Abelian simple
groups and $\overline{G}\leqslant {\rm Aut}(S)$. It is clear that
$P_i\in \mathcal{S}_{7}$. Moreover, by Lemma \ref{elementary},
$|P_i|$ is divisible by 3 and $\pi({\rm Out}(P_i))\subseteq
\{2,3\}$.

First, we claim that the order of $\overline{G}$ is divisible by
$7$. Suppose false. Then, $7$ divides $|N|$ and also
$\pi(S)=\pi(P_i)=\{2,3,5\}$. Thus, $P_i$ can only be isomorphic to
$\mathbb{A}_5$, $\mathbb{A}_6$ or $U_4(2)$. In any case,
$\overline{G}$ contains the Frobenius subgroup $\mathbb{A}_4$ of
order $4\cdot 3$, and by Lemma \ref {vicmazurov}, $G$ has an
element of order $21$, which is a contradiction. Therefore,
$|\overline{G}|$ is divisible by $7$.

Next, we will show that the order of $S$ is also divisible by $7$.
Suppose not. In this case, again, $P_i$ can only be isomorphic to
$\mathbb{A}_5$, $\mathbb{A}_6$ or $U_4(2)$. By previous paragraph,
$7$ divides the order of ${\rm Out}(S)$. But ${\rm Out}(S)={\rm
Out}(S_1)\times \cdots\times {\rm Out}(S_m)$, where the groups
$S_j$ are the direct products of those $P_i$ which are isomorphic
and $S\cong S_1\times \cdots \times S_m$. Therefore, for some $j$,
$7$ divides the order of an outer automorphism group of a direct
product $S_j$ of $t$ isomorphic simple groups $P_i$. By Lemma
\ref{elementary}, ${\rm Out}(P_i)$ is not divisible by $p$ . By
Lemma \ref{automorphism}, $|{\rm Aut}(S_j)|=|{\rm
Aut}(P_i)|^t\cdot t!$. Therefore, $t\geqslant 7$ and $S_j$ admits
an automorphism of order $7$ which centralizes the element
$(a,a,\ldots,a)$ of $S_j$ of order $3$, where $a$ is any element
of $P_i$ of order $3$. Thus $3\cdot 7\in \omega(G)$, a
contradiction.

Finally, it is easy to see that $k=1$, since otherwise $7\cdot
3\in \omega(S)\subseteq \omega(G)$; a contradiction. The lemma is
proved.
\end{proof}

Below, we list the simple groups $S\in \mathcal{S}_7$ satisfy
$7\in \pi(S)$:
$$ \mathbb{A}_n, n=7, 8, 9, 10; \ J_2; \ L_2(7),
 \ L_2(8), \ L_2(49), \ L_3(4), \ U_3(3), \ U_3(5), \ U_4(3), \ S_6(2), \ S_4(7), \ O_8^+(2).$$
\begin{lm}\label{l3janko} $S\cong J_2$ or $\mathbb{A}_8$.
\end{lm}
\begin{proof} Since $G$ has no element of order $9$, $20$ or $25$, $S$ can not be
isomorphic to $\mathbb{A}_9$, $\mathbb{A}_{10}$, $L_2(8)$,
$U_4(3)$, $S_6(2)$, $O_8^+(2)$, $L_2(49)$, $U_5(3)$, or $S_4(7)$.
Therefore, $S$ can only be
$$ \mathbb{A}_n, n=7, 8; \ J_2; \ L_2(7),
 \ L_3(4) \ {\rm or} \ U_3(3).$$
We consider above possibilities for the group $S$ consecutively.

First, Suppose that $S$ is isomorphic to $\mathbb{A}_{7}$. Then,
since ${\rm Aut}(\mathbb{A}_7)=\mathbb{S}_7$ does not contain an
element of order $15$ and $G/N\leqslant {\rm
Aut}(\mathbb{A}_{7})$, hence $3\in \pi(N)$ or $5\in \pi(N)$. On
the other hand, since $S$ contains Frobenius subgroups $4:3$ and
$3^2:4$, it follows by Lemma \ref{vicmazurov} that $9\in
\omega(G)$ or $20\in \omega(G)$, which is a contradiction.

Next, we assume that $S$ is isomorphic to $L_2(7)$ or $U_3(3)$.
Since $S$ contains a Frobenius subgroup of order $4\cdot 3$, for
any prime $p$ of the order of $N$, by Lemma \ref{vicmazurov}, $G$
has an element of order $3\cdot p$. This means that $N$ is a
$\{2,5\}$-group. Moreover, since $15, 24\in \omega(G)\backslash
\omega({\rm Aut}(S))$, it follows that $|N|=2^\alpha \cdot
5^\beta$ where $\alpha, \beta \in \mathbb{N}$.

Finally, we assume that $S$ is isomorphic to  $L_3(4)$. Since $S$
contains Frobenius subgroups of orders $7\cdot 3$ and $16\cdot
5$, and $G$ does not contain an element of order $9$, $25$ or
$35$, $N$ is a $2$-group. But then $G/N$ and so ${\rm
Aut}(L_3(4))$ contains an element of order $15$, which is a
contradiction.
\end{proof}
\begin{lm}\label{lemmaeight}
If $S\cong \mathbb{A}_8$, then $N$ is a $2$-subgroup and
$G/N\cong \mathbb{A}_8$.
\end{lm}
\begin{proof}
By Lemma \ref{l2janko}, $G/N\cong \mathbb{A}_8$ or $G/N\cong {\rm
Aut}(\mathbb{A}_8)=\mathbb{S}_8$. First of all, since $S$ contains
Frobenius subgroups $4:3$ and $3^2:4$, it follows by Lemma
\ref{vicmazurov} that $N$ is a $2$-group. Moreover, the case when
$G/N \simeq \mathbb{S}_8$ is impossible. In contrary case, let
$V$ be a $G$-chief factor of $N=O_2(G)$. By \cite{BrAtl},
$\dim(V)=6, 8$ or $40$, since in other cases $C_G(x)$, where $x$
is an element of order 15, contains an involution and $G$ should
contain an element of order 30 which is not the case. If
$\dim(V)=6$ or $40$, then, using GAP and the corresponding
information from \cite{ar}, it is easy to calculate that some
element $a$ of order 10 in $G/N$ induces   in $V$ by conjugation
a linear transformation $t$ such that $1+t+t^2+\cdots+t^9\ne 0$
and hence, by Lemma \ref{lemma0}, $G$ contains an element of
order 20 which is impossible. Thus, $\dim(V)=8$ and hence an
element $x\in G/N$ from the class $3A$ acts on $N$
fixed-point-freely. Let $L\simeq L_2(4)$ be a subgroup of $G/N$
containing $x$. By Higman's theorem about action of $L_2(2^m)$ on
a $2$-group with an element of order $3$ acting
fixed-point-freely (see \cite[Theorem 8.2]{Higman}), $N$ is
elementary Abelian, hence $G$ cannot contain an element of order
$24$. This contradiction proves the lemma.
\end{proof}
\begin{lm}\label{new2007}
If $S\cong J_2$, then $G\cong {\rm Aut}(J_2)$.
\end{lm}
\begin{proof}
First of all, we prove that $N$ is a non-trivial 2-group. To prove
this, we observe that:\\

(1) {\em If $H/N\cong J_2$ with $\omega(H)\subseteq \omega({\rm
Aut}(J_2))$ and $N$ is a non-trivial $p$-group, then $p=2$ and
$C_H(N)\subseteq N$.}\\

\begin{proof}
Suppose that $p\neq 2$. Then $p=3, 5$ or 7 and one can assume that
$N$ is elementary Abelian $p$-group. Suppose first that $p=7$.
Then the table of $7$-Brauer characters of $J_2$ (see \cite{gap})
shows that an element from the class $3a$ in $J_2$ centralizes an
element of order $3$ in $N$ and hence $H$ contains an element of
order $21$ which is not the case. If $p=3$ or $p=5$, then the
corresponding table of $p$-Brauer characters shows that $H$
contains an element of order $7\cdot p$ which is not the case.
Thus $p=2$. If $C_H(N)\nsubseteq N$, then $C_H(N)N=H$ and hence
$C_H(N)$ contains an element of order $30$ which is not the case.
The lemma is proved. \end{proof}\\

(2) {\em If $H/N\cong J_2$ and $\omega(H)\subseteq \omega({\rm
Aut}(J_2))$, then $N$ is a $2$-group.}\\

\begin{proof} Suppose false. If $H$ is a minimal counter-example,
then, by part (1) and Frattini argument, $N$ is an extension of a
non-trivial elementary Abelian $p$-group $P$ of odd order by a
non-trivial $2$-group and $C_H(P)=P$.

It is obvious that $H/N$ acts on the center $Z$ of $2$-group
$N/P$. If $H/N$ acts trivially on $Z$ then $H/P$ contains an
element of order $30$ which is not the case. Thus $H/N$ acts
faithfully, and hence $ZS$, where $S$ is a subgroup of order $7$
in $H/N$, contains a Frobenius subgroup $[Z,S]\cdot S$. Then the
full pre-image of this group contains an element of order $7\cdot
p$, which is not the case. \end{proof}

Therefore, $N$ is a 2-group and by Lemma \ref{l2janko}, $G/N\cong
J_2$ or $G/N\cong {\rm Aut}(J_2)$. If $G/N\cong J_2$, then $N\neq
1$, and since $G$ cannot contain an element of order $30$, the
table of $2$-Brauer characters shows that an element of order $7$
from $G$ acts on $N$ fixed-point-freely and hence $G$ cannot
contain an element of order $14\in \omega({\rm Aut}(J_2))$. Thus
$G/N\cong {\rm Aut}(J_2)$.

If $N\neq 1$, then one can assume that $N$ is elementary Abelian
and $N$ is an absolutely irreducible ${\rm Aut}(J_2)$-module over
a field of characteristic $2$. Since $G$ contains an element of
order $30$, the table of $2$-Brauer characters of ${\rm Aut}(J_2)$
shows that $N$ is of dimension $12$. Using Lemma \ref{lemma0} and
the information on ${\rm Aut}(J_2)$ from \cite{ar}, it is easy to
check that some pre-image in $G$ of an element contained in the
class $10A$ of ${\rm Aut}(J_2)$ is of order $20\notin \omega({\rm
Aut}(J_2))$.
\end{proof}

Authors: \vspace{2mm}

{\bf V. D. Mazurov},

{\em Sobolev Institute of Mathematics},

{\em Novosibirsk, $630090$, Russia},

and

{\em Novosibirsk State University},

\verb"mazurov@math.nsc.ru"

\vspace{3mm}

{\bf A. R. Moghaddamfar},

{\em Faculty of Mathematics,

K. N. Toosi University of Technology,

P. O. Box $16315$--$1618$, Tehran, Iran},

\verb"moghadam@kntu.ac.ir"   \ \  and \ \ \verb"moghadam@ipm.ir"


\begin{thebibliography}{99}
\bibitem{atlas} J. H. Conway, R. T. Curtis, S. P. Norton, R. A. Parker,
and R. A. Wilson, {\em Atlas of finite groups}, Clarendon Press,
Oxford, 1985.

\bibitem{Higman} G. Higman, {\em Odd characterizations of finite simple groups},
(Lecture notes, University of Michigan, 1968).

\bibitem{BrAtl} C. Jansen, K. Lux, R. A. Parker and R. A. Wilson,
{\em An atlas of Brauer characters}, Clarendon Press, Oxford,
1995.

\bibitem{k} A. S. Kondrat\'ev, {\em On prime graph components of
finite simple groups}, Math. USSR Sbornik, 180(6)(1989), 787-797.

\bibitem{lishi} H. Li and W. J. Shi, {\em A characterization of
some sporadic simple groups}, Chinese Ann. Math. 14A(2)(1993),
144-151. (in Chinese)

\bibitem{lm} M. S. Lucido and A. R. Moghaddamfar, {\em  Groups with complete
prime graph connected components}, J. Group Theory, 7(3)(2004),
373-384.

\bibitem{mazu} V. D. Mazurov,
{\em The set of orders of elements in a finite group}, Algebra and
Logic, 33(1)(1994), 49-56.

\bibitem{maz} V. D. Mazurov,
{\em Characterizations of finite groups by sets of element
orders}, Algebra and Logic, 36(1)(1997), 23-32.

\bibitem{mazshi} V. D. Mazurov and W. J. Shi, {\em A note to the
characterization of sporadic simple groups}, Alg. colloq.
5(3)(1998), 285-288.

\bibitem{newmaz} V. D. Mazurov, {\em Recognition of finite simple groups $S_4(q)$ by their
element orders}, Algebra and Logic, 41(2)(2002), 93-110.

\bibitem{maz04} V. D. Mazurov, {\em Characterization of groups by
arithmetic properties}, Algebra Colloquium, 11(1)(2004), 129-140.

\bibitem {mzd} A. R. Moghaddamfar, A. R. Zokayi
and M. R. Darafsheh, {\em On characterizability of the
automorphism groups of sporadic simple groups by their element
orders}, Acta Mathematica Sinica, 20(4)(2004), 653-662.

\bibitem{praegershi} C. E. Praeger and W. J. Shi, {\em A characterization of some
alternating and symmetric groups},
Comm. Algebra, 22(5)(1994), 1507-1530.

\bibitem{gap} M. Sch$\rm \ddot{o}$nert et al., {\em GAP- Groups,
Algorithms and Programming,} (2004),
http://www-gap.mcs.st-and.ac.uk/

\bibitem{wujie1} W. J. Shi, {\em A charateristic property
of $J_1$ and ${\rm PSL}_2(2^n)$}, Adv. in Math. 16(1987),
397-401. (in Chinese)

\bibitem{wujie2} W. J. Shi, {\em A charateristic property of Mathieu groups},
 Chinese Ann. Math. 9A(5)(1988) 575-580. (in Chinese)

\bibitem{wujie3} W. J. Shi, {\em A charateristic property of Conway simple group $Co_2$},
 Chinese J. Math. 9(2)(1989), 171-172, (in Chinese)

\bibitem{wujie4} W. J. Shi, {\em A characterization of the Higman-Sims simple group},
 Houston J. Math. 16(1990), 597-602.

\bibitem{wujie5} W. J. Shi, {\em The characterization of the sporadic simple groups
 by their element orders}, Algebra Colloq. 1(2)(1994), 159-166.

\bibitem{wujieli} W. J. Shi and H. Li,
 {\em A charateristic property of $M_{12}$ and ${\rm PSU}(6,2)$},
 Acta Math. Sinica 32(6)(1989), 758-764. (in Chinese)

\bibitem{vasilev} A. V. Vasilev, {\em On recognition of all finite
nonabelian simple groups with orders having prime divisors at
most $13$}, Sib. Math. J., 46(2)(2005), 246-253.

\bibitem {w} J. S. Williams, {\em  Prime graph components of finite
groups}, J. of Algebra,  69(2)(1981), 487-513.

\bibitem{ar} R. A. Wilson, {\em ATLAS of finite group
representations}, http://web.mat.bham.ac.uk/atlas/

\bibitem {zavarnitsine} A. V. Zavarnitsine, {\em  Recognition of alternating group of
degrees $r+1$ and $r+2$ for primes $r$ and the group of degree
$16$ by their element order sets},  Algebra and Logic,
39(6)(2000), 370-377.
\end{thebibliography}
\end{document}